\documentclass[oneside,english]{amsart}
\usepackage[T1]{fontenc}
\usepackage[latin9]{inputenc}
\usepackage{geometry}
\usepackage{amsthm}
\usepackage{amssymb}
\usepackage{graphicx}
\usepackage{esint}

\makeatletter
\numberwithin{equation}{section}
\numberwithin{figure}{section}
\theoremstyle{plain}
\newtheorem{thm}{\protect\theoremname}
\theoremstyle{plain}
\newtheorem{lem}[thm]{\protect\lemmaname}
\theoremstyle{definition}
\newtheorem{defn}[thm]{\protect\definitionname}
\theoremstyle{remark}
\newtheorem{rem}[thm]{\protect\remarkname}

\usepackage{todonotes}
\global\long\def\Im{\operatorname{Im}}

\global\long\def\Arg{\operatorname{Arg}}

\global\long\def\Disc{\operatorname{Disc}}

 \author{Ashish Goswami}

 \address{Saratoga High School, Saratoga, California, USA.} 
 \email{ashish.keebab@gmail.com}

 \author{Khang Tran}
 \address{California State University, Fresno, California, USA}
 \email{khangt@mail.fresnostate.edu}  

\makeatother

\usepackage{babel}
\providecommand{\definitionname}{Definition}
\providecommand{\lemmaname}{Lemma}
\providecommand{\remarkname}{Remark}
\providecommand{\theoremname}{Theorem}

\begin{document}
\title{A class of polynomials from enumerating queen paths}
\begin{abstract}
We study a class polynomials obtained from an enumeration of the number
of queen paths. In particular, we find the generating function for
the diagonal sequence of this table and the zero distribution of a
sequence of related polynomials.
\end{abstract}

\subjclass[2000]{30C15; 26C10; 11C08}
\keywords{Zero distribution, Generating function}
\maketitle

\section{Introduction}

In this paper, we study a class of polynomials arising from a combinatorial
problem of counting the number of rook and queen paths. We let $a_{m,n}$
and $b_{m,n}$ be the numbers of paths a rook and queen (respectively)
can move from $(0,0)$ to $(m,n)$ on an infinite 2D chess board.
These numbers have been studied in \cite{eft} and they satisfy the
recurrences
\begin{equation}
a_{m,n}=2a_{m-1,n}+2a_{m,n-1}-3a_{m-1,n-1}\label{eq:rookrecurrence}
\end{equation}
and 
\begin{align}
b_{m,n} & =2b_{m-1,n}+2b_{m,n-1}-b_{m-1,n-1}-3b_{m-2,n-1}\nonumber \\
 & -3b_{m-1,n-2}+4b_{m-2,n-2}.\label{eq:queenrecurrence}
\end{align}

Motivated by these recurrences, we define a table of polynomials by
replacing one of the coefficients in the equations by a variable $z$.
In particular, we defined the table of rook polynomials $\left\{ P_{m,n}(z)\right\} _{m,n=0}^{\infty}$
by the recurrence 
\[
P_{m,n}(z)=2P_{m-1,n}(z)+2P_{m,n-1}(z)+zP_{m-1,n-1}(z)
\]
for $m,n\in\mathbb{N}$ and $(m,n)\ne(0,0)$. For simplicity, we use
the standard initial condition $P_{0,0}(z)=1$ and $P_{m,n}(z)=0$
if $m<0$ or $n<0$. In the definition above, we replace the coefficient
of $a_{m-1,n-1}$ in (\ref{eq:rookrecurrence}) by $z$ since we will
see below that the main diagonal polynomials, $P_{m,m}(z)$, have
a connection with the famous sequence of Legendre polynomials.To see
this connection, we note from the given recurrence relation and the
initial condition that the polynomials $\left\{ P_{m,n}(z)\right\} _{m,n=0}^{\infty}$
are generated by 
\[
\sum_{m=0}^{\infty}\sum_{n=0}^{\infty}P_{m,n}(z)s^{m}t^{n}=\frac{1}{1-2s-2t-zst}.
\]
With the substitutions $s\rightarrow s/(-2)$, $t\rightarrow t/(-2)$,
and $z\rightarrow-4z$, we have 
\[
\sum_{m=0}^{\infty}\sum_{n=0}^{\infty}\frac{P_{m,n}(-4z)}{(-2)^{m+n}}s^{m}t^{n}=\frac{1}{1+s+t+zst}.
\]
From \cite[Lemma 4]{lt}, we conclude that when $m=n$
\[
\frac{P_{m,m}(-4z)}{2^{2m}}=z^{m}L_{m}\left(\frac{2}{z}-1\right)
\]
or equivalently 
\begin{equation}
P_{m,m}(z)=(-z)^{m}L_{m}\left(-\frac{8}{z}-1\right)=z^{m}L_{m}\left(\frac{8}{z}+1\right)\label{eq:rookLegendre}
\end{equation}
where $L_{m}(z)$ is the sequence of Legendre polynomials generated
by 
\[
\sum_{m=0}^{\infty}L_{m}(z)t^{m}=\frac{1}{(1-2zt+t^{2})^{1/2}}.
\]
The sequence of Legendre polynomials is a special case of the sequence
of Gegenbauer polynomials whose generating function is (\cite[IV.2]{sw})
\[
\frac{1}{(1-2zt+t^{2})^{\alpha}}.
\]
In the case $\alpha=1$, the function above generates the sequence
of Chebyshev polynomials of the second kind. For $\alpha>-1/2$, the
sequence of Gegenbauer polynomials are orthogonal on $[-1,1]$ (\cite[page 302]{aar})
with respect the weight function $(1-z^{2})^{\alpha-1/2}$. As a consequence,
the zeros of Gegenbauer polynomials lie on this interval for $\alpha>-1/2$.
We deduce from (\ref{eq:rookLegendre}) that the zeros of $P_{m,m}(z)$
lie on the interval $(-\infty,-4]$.

In a similar way, motivated by (\ref{eq:queenrecurrence}), we define
the table of Queen polynomials by the recurrence 
\begin{align}
Q_{m,n}(z) & =2Q_{m-1,n}(z)+2Q_{m,n-1}(z)-Q_{m-1,n-1}(z)-3Q_{m-2,n-1}(z)\nonumber \\
 & -3Q_{m-1,n-2}(z)+zQ_{m-2,n-2}(z).\label{eq:queenpolycurrence}
\end{align}
and the standard initial condition $Q_{0,0}(z)=1$ and $Q_{m,n}(z)=0$
if $m<0$ or $n<0$. Equivalently, this table is generated by 
\begin{equation}
\sum_{m=0}^{\infty}\sum_{n=0}^{\infty}Q_{m,n}(z)s^{n}t^{m}=\frac{1}{1-2(s+t+st)+3(st+s^{2}t+st^{2})-zs^{2}t^{2}}.\label{eq:queenpolygen}
\end{equation}
This means that for each $z\in\mathbb{C}$, there is a sufficiently
small $\delta>0$ so that (\ref{eq:queenpolycurrence}) holds for
all $|s|<\delta$ and $|t|<\delta$. Similar to the table of rook
polynomials above, we seek to understand the generating function for
the diagonal sequence $Q_{m,m}(z)$ and the zero distribution of related
polynomials. To achieve this goal, we first make the substitution
$s=x/t$ in (\ref{eq:queenpolygen}) to obtain the following 
\begin{align}
\sum_{m=0}^{\infty}\sum_{n=0}^{\infty}Q_{m,n}(z)x^{n}t^{m-n} & =\frac{1}{1-2(x/t+t+x)+3(x+x^{2}/t+xt)-zx^{2}}\nonumber \\
 & =\frac{t}{t^{2}(3x-2)+t(-zx^{2}+x+1)+3x^{2}-2x}.\label{eq:queenpolygensub}
\end{align}
Since $|s|<\delta$ and $|t|<\delta$, we have $\frac{|x|}{\delta}<|t|<\delta.$
We deduce from the equation above that the generating function 
\begin{equation}
\sum_{m=0}^{\infty}Q_{m,m}(z)x^{m}\label{eq:genqueendiagonal}
\end{equation}
 is the $t^{0}$-coefficient of the Laurent series of (\ref{eq:queenpolygensub})
in the annulus 
\begin{equation}
\frac{|x|}{\delta}<|t|<\delta.\label{eq:annulus}
\end{equation}
 To compute this coefficient, we apply partial fraction decomposition
to write (\ref{eq:queenpolygensub}) as 
\begin{equation}
\frac{\tau_{1}}{(3x-2)(\tau_{1}-\tau_{2})}\cdot\frac{1}{t-\tau_{1}}+\frac{\tau_{2}}{(3x-2)(\tau_{2}-\tau_{1})}\cdot\frac{1}{t-\tau_{2}}.\label{eq:sequence}
\end{equation}
where $\tau_{1}$ and $\tau_{2}$ are the two zeros (in $t$) of the
denominator of (\ref{eq:queenpolygensub}). The quadratic formula
gives
\[
\tau_{1}=\frac{zx^{2}-x-1+\sqrt{(-zx^{2}+x+1)^{2}-4x(3x-2)^{2}}}{2(3x-2)}
\]
and 
\[
\tau_{2}=\frac{zx^{2}-x-1-\sqrt{(-zx^{2}+x+1)^{2}-4x(3x-2)^{2}}}{2(3x-2)}.
\]
From (\ref{eq:annulus}), we have $|x|<\delta^{2}$. Thus for small
$\delta$, $x$is small and $\tau_{1}\sim x/4$ and $\tau_{2}\sim1/2$.
We conclude that for sufficiently small $\delta$, 
\[
\tau_{1}<|t|<\tau_{2}
\]
for any $t$ in the annulus (\ref{eq:annulus}). With these inequalities,
we apply the Laurent series expansions 
\begin{align*}
\frac{1}{t-\tau_{1}} & =\frac{1}{t}\frac{1}{1-\tau_{1}/t}=\sum_{n=0}^{\infty}\frac{\tau_{1}^{n}}{t^{n+1}},\\
\frac{1}{t-\tau_{2}} & =\frac{1}{\tau_{2}}\frac{1}{t/\tau_{2}-1}=-\sum_{n=0}^{\infty}\frac{t^{n}}{\tau_{2}^{n+1}},
\end{align*}
to conclude that (\ref{eq:genqueendiagonal}), which the $t^{0}$-coefficient
in the Laurent series expansion of \ref{eq:sequence}, is 
\begin{align*}
-\frac{1}{(3x-2)(\tau_{2}-\tau_{1})} & =\frac{1}{\sqrt{(-zx^{2}+x+1)^{2}-4x(3x-2)^{2}}}\\
 & =\frac{1}{\sqrt{x^{4}z^{2}+x^{3}(-2z-36)+x^{2}(49-2z)-14x+1}}.
\end{align*}
Similar to the idea that Gegenbauer polynomials are generalization
of Legendre polynomial, the generating function above motivates us
to define a sequence of polynomials 
\begin{equation}
\sum_{m=0}^{\infty}P_{m}(z)t^{m}=\frac{1}{(t^{4}z^{2}+t^{3}(-2z-36)+t^{2}(49-2z)-14t+1)^{\alpha}}.\label{eq:generalgen}
\end{equation}
We conjecture that for any $\alpha>0$, the zeros of $P_{m}(z)$ lie
on the interval $(-\infty,-9/4)$. In the next section, we will show
that the conjecture holds for $\alpha=1$. The method of the proof
in this paper provides us a direction in tackling the problems of
finding the zero distribution of sequence of polynomials whose denominator
of the generating function is nonlinear in $z$. For studies on the
case this denominator is linear in $z$ in its variations, see \cite{tz,ft}.

\section{Zero distribution of $P_{m}(z)$}

The main goal of this section is to prove the theorem below. 
\begin{thm}
\label{thm:maintheorem} The zeros of the polynomials $P_{m}(z)$
generated by 
\begin{equation}
\sum_{m}P_{m}(z)t^{m}=\frac{1}{t^{4}z^{2}+t^{3}(-2z-36)+t^{2}(49-2z)-14t+1}\label{eq:Pmgen}
\end{equation}
lie on the interval $(-\infty,-9/4)$. 
\end{thm}

To prove Theorem \ref{thm:maintheorem}, we will count the number
of zeros of $P_{m}(z)$ on $(-\infty,-9/4)$ and show that this number
is at least the degree of $P_{m}(z)$. Theorem \ref{thm:maintheorem}
will follow directly from the Fundamental Theorem of Algebra. The
theorem below provides an upper bound for the degree of $P_{m}(z)$. 
\begin{lem}
\label{lem:degree}The degree of $P_{m}(z)$ is at most $\lfloor\frac{m}{2}\rfloor$. 
\end{lem}

\begin{proof}
From (\ref{eq:Pmgen}), the sequence $\left\{ P_{m}(z)\right\} _{m=0}^{\infty}$
satisfies the recurrence relation 
\[
P_{m}(z)-14P_{m-1}(z)+(49-2z)P_{m-2}(z)+(-2z-36)P_{m-3}(z)+z^{2}P_{m-4}(z)=0
\]
for $m\ge1$ with initial condition $P_{0}(z)=1$ and $P_{-m}=0$.
From the recurrence above and the induction hypothesis, we have 
\[
\deg(P_{m}(z))\le\max\left(\Bigl\lfloor\frac{m-1}{2}\Bigr\rfloor,\Bigl\lfloor\frac{m-2}{2}\Bigr\rfloor+1,\Bigl\lfloor\frac{m-3}{2}\Bigr\rfloor+1,\Bigl\lfloor\frac{m-4}{2}\Bigr\rfloor+2\right)=\Bigl\lfloor\frac{m}{2}\Bigr\rfloor.
\]
\end{proof}
It remains show that number of zeros of $P_{m}(z)$ on $(-\infty,-9/4)$
is at least $\lfloor\frac{m}{2}\rfloor$. For this reason, we assume
$z\in(-\infty,-9/4)$. For each $z\in(-\infty,-9/4)$, let $t_{1}$,
$t_{2}$, $t_{3}$, and $t_{4}$ be the zeros (in $t$) of 
\begin{equation}
D(t,z):=t^{4}z^{2}+t^{3}(-2z-36)+t^{2}(49-2z)-14t+1.\label{eq:denominator}
\end{equation}
We will show that for $z\in(-\infty,-9/4)$, these zeros are not real.
A useful concept in proving this is the discriminant of a polynomial
defined below. 
\begin{defn}
The discriminant of a polynomial $P(z)$ with lead coefficient $p$
and degree $n$ is 
\[
\Disc_{z}P(z)=p^{2n-2}\prod_{1\le i<j\le n}(z_{i}-z_{j})^{2}
\]
where $z_{i}$, $1\le i\le n$, are the zeros of $P(z)$. 
\end{defn}

From this definition, $\Disc_{z}P(z)=0$ if and only if $P(z)$ has
a multiple zero. In the case, the degree of $P(z)$ is 4, $\Disc_{z}P(z)>0$
if and only if either all the zeros of $P(z)$ are real or none of
these zeros are real \cite{lt}. For further studies of discriminants
of various polynomials, see \cite{apostol,ds,gi}.
\begin{lem}
\label{lem:distinct} For each $z\in(-\infty,-9/4)$, the four zeros
of $D(t,z)$ are distinct and not real. 
\end{lem}

\begin{proof}
From a computer algebra system, the discriminant of $D(t,z)$ as a
polynomial in $t$ is 
\[
\Disc_{t}D(t,z)=-256(z-4)(4z-15)^{2}(4z+9)^{2}>0
\]
for $z\in(-\infty,-9/4)$. Thus either (i) all zeros of $D(t,z)$
are real for all $z\in(-\infty,-9/4)$ or (ii) none of zeros of $D(t,z)$
is real for all $z\in(-\infty,-9/4)$. When $z=-3$, from (\ref{eq:denominator})
we can check that the following polynomial 
\[
D(t,-3)=9t^{4}-30t^{3}+55t^{2}-14t+1
\]
has four non-real zeros. Thus all the zeros of $D(t,z)$ are non-real
for all $z\in(-\infty,-9/4)$. 
\end{proof}
Since $D(t,z)$ is a real polynomial (for $z\in(-\infty,-9/4)$),
their non-real zeros, denoted by $t_{1}$, $t_{2}$, $t_{3}$, and
$t_{4}$, form conjugate pairs. Without loss of generality, we let
$t_{1}=\overline{t_{2}}$, $t_{3}=\overline{t_{4}}$, $|t_{1}|\le|t_{3}|$,
and $t_{1}$ and $t_{3}$ lie on the upper half plane. We can write
these zeros as $t_{1}=re^{i\theta}$, $t_{2}=re^{-i\theta}$, $t_{3}=\varrho e^{i\phi}$,
and $t_{4}=\varrho e^{-i\phi}$ where $0<r\le\rho$. These zeros satisfy
the Vieta's formulas

\begin{align*}
t_{1}+t_{2}+t_{3}+t_{4} & =\frac{2z+36}{z^{2}},\\
t_{1}t_{2}+t_{1}t_{3}+t_{1}t_{4}+t_{2}t_{3}+t_{2}t_{4}+t_{3}t_{4} & =\frac{-2z+49}{z^{2}},\\
t_{1}t_{2}t_{3}+t_{1}t_{2}t_{4}+t_{1}t_{3}t_{4}+t_{2}t_{3}t_{4} & =\frac{14}{z^{2}},\\
t_{1}t_{2}t_{3}t_{4} & =\frac{1}{z^{2}}.
\end{align*}
The first equation is the same as 
\begin{equation}
2r\cos\theta+2\varrho\cos\phi=\frac{2z+36}{z^{2}}.\label{eq:firstsym}
\end{equation}
Similarly the second equation is equivalent to 
\[
r^{2}+\varrho^{2}+r\varrho e^{-i(\theta+\phi)}+r\varrho e^{i(\theta-\phi)}=\frac{-2z+49}{z^{2}}
\]
where the left side is 
\[
r^{2}+\varrho^{2}+2r\varrho\cos(\theta+\phi)+2r\varrho\cos(\theta-\phi)=r^{2}+\varrho^{2}+4r\varrho\cos\theta\cos\phi.
\]
The third equation gives

\[
2r^{2}\varrho\cos\phi+2r\varrho^{2}\cos\theta=\frac{14}{z^{2}}.
\]
Using 
\[
r\varrho=-\frac{1}{z}
\]
from the fourth equation (as $z<0$) we can rewrite this equation
as 
\begin{equation}
2r\cos\phi+2\varrho\cos\theta=\frac{-14}{z}.\label{eq:thirdfourthsym}
\end{equation}
Recall that $|t_{1}|\le|t_{3}|$. From these elementary symmetric
equations, we can show that this inequality is strict in the lemma
below. 
\begin{lem}
\label{lem:zerostrictineq} For any $z\in(-\infty,-9/4)$, we have
$|t_{1}|=|t_{2}|<|t_{3}|=|t_{4}|$. 
\end{lem}

\begin{proof}
If, by the way of contradiction, that $r=\rho$, then (\ref{eq:firstsym})
and (\ref{eq:thirdfourthsym}) yield 
\[
\frac{2z+36}{z^{2}}=\frac{-14}{z}
\]
or equivalently 
\[
16z^{2}+36z=0
\]
which is a contradiction since $z\in(-\infty,-9/4).$ 
\end{proof}
Recall that for each $z\in(-\infty,-9/4)$, $\theta$ is the principal
angle of $t_{1}=t_{1}(z)$. Thus we can view $\theta$ as a function
of $z\in(-\infty,-9/4)$. 
\begin{lem}
If $t_{1}=re^{i\theta}$, then $\theta(z)$ is a decreasing function
on $z\in(-\infty,-9/4)$. 
\end{lem}

\begin{proof}
By the chain rule, 
\[
\frac{d\theta}{dz}=\frac{dt_{1}}{dz}\cdot\frac{d\theta}{dt_{1}}.
\]
We will show that $d\theta/dz\ne0$ by showing that each term on the
right side is nonzero. We differentiate both sides of 
\[
t_{1}^{4}z^{2}+t_{1}^{3}(-2z-36)+t_{1}^{2}(49-2z)-14t_{1}+1=0
\]
with respect to $z$ and obtain 
\[
\frac{dt_{1}}{dz}=-\frac{D_{z}(t_{1},z)}{D_{t}(t_{1},z)}=-\frac{2zt_{1}^{4}-2t_{1}^{3}-2t_{1}^{2}}{4t_{1}^{3}z^{2}+3t_{1}^{2}(-2z-36)+2t_{1}(49-2z)-14}.
\]
We note that the denominator of the last expression is nonzero since
$D_{t}(t_{1},z)\ne0$ as $t_{1}$ is a simple zero of $D(t,z)$ by
Lemma \ref{lem:distinct}. Thus $dt_{1}/dz$ is a continuous function
in $z\in(-\infty,-9/4)$. We claim that the numerator of this expression,
\[
2t_{1}^{2}(zt_{1}^{2}-t_{1}-1),
\]
is also nonzero. Indeed, if by contradiction $zt_{1}^{2}-t_{1}-1=0$.
Then, $z=\frac{t_{1}+1}{t_{1}^{2}}$. We substitute this value of
$z$ into $D(t_{1},z)=0$ and conclude 
\[
t_{1}^{4}z^{2}+t_{1}^{3}(-2z-36)+t_{1}^{2}(49-2z)-14t_{1}+1=-4t_{1}(-2+3t_{1})^{2}=0
\]
which is a contradiction since $t_{1}\notin\mathbb{R}$. Thus, we
can conclude $\frac{dt_{1}}{dz}\ne0$ for $z\in(-\infty,-9/4)$.

We will show that $d\theta/dz$ is also nonzero for $z$ in this interval.
We differentiate both sides of $t_{1}=re^{i\theta}$ with respect
to $\theta$ and conclude that 
\[
\frac{dt_{1}}{d\theta}=e^{i\theta}(\frac{dr}{d\theta}+ir),
\]
or equivalently 
\[
\frac{d\theta}{dt_{1}}=\frac{1}{e^{i\theta}(\frac{dr}{d\theta}+ir)}\ne0.
\]
Note that the denominator of the last expression is nonzero since
$dr/d\theta\in\mathbb{R}$ and thus $d\theta/dt_{1}$ is continuous
on $z\in(-\infty,-9/4)$.

Since 
\[
\frac{d\theta}{dz}=\frac{dt_{1}}{dz}\cdot\frac{d\theta}{dt_{1}}
\]
is a continuous function in $z\in(-\infty,-9/4)$ and it has no zero
on this interval, $\theta(z)$ is monotone on $z\in(-\infty,-9/4)$.
Thus, to complete this lemma, we compare two values of $\theta(z)$
at two different values of $z$. From a simple computer algebra, we
have $\theta(-10)=0.55491..$ and $\theta(-3)=0.206599...$ from which
the lemma follows.
\end{proof}
Now that we know $\theta(z)$ is decreasing, the lemma below provides
the image of $(-\infty,-9/4)$ under this map. 
\begin{lem}
\label{lem:decreasing} The function $\theta(z)$ maps $(-\infty,-9/4)$
onto $(0,\pi/2)$. 
\end{lem}

\begin{proof}
Recall that $t_{1}(z)$ is a zero of 
\[
D(t,z)=t^{4}z^{2}+t^{3}(-2z-36)+t^{2}(49-2z)-14t+1.
\]
A simple evaluation yields 
\[
\lim_{z\to-\frac{9}{4}}D(t,z)=\frac{1}{16}(4-28t+9t^{2})^{2}
\]
where the zeros of the last expression are positive real. Thus 
\[
\lim_{z\rightarrow-9/4}\theta(z)=\lim_{z\rightarrow-9/4}\Arg t_{1}=0.
\]
Next, we consider the case $z\rightarrow-\infty$. First, we note
that as $z$ approaches $-\infty$, $t_{1}$ must approach 0, since
if otherwise the modulus of the term $t_{1}^{4}z^{2}$ in $D(t_{1},z)$
is larger than those of all the terms of $D(t_{1},z)$ when $|z|$
is large. This contradicts to $D(t_{1},z)=0$.

We next show that $\lim_{z\rightarrow-\infty}|t_{1}^{3}z|$ exists
and equals to $0$ by showing that 

\[
\limsup_{z\rightarrow-\infty}|t_{1}^{3}z|=0
\]
where the left side of the expression above is defined as 
\[
\lim_{x\rightarrow-\infty}\sup\{|t_{1}^{3}z|:z\in(-\infty,x)\}.
\]
If, by contradiction, 
\[
\limsup_{z\rightarrow-\infty}|t_{1}^{3}z|\ne0
\]
then we have 
\[
\limsup_{z\rightarrow-\infty}|t_{1}z|=\limsup_{z\rightarrow-\infty}\left|\frac{t_{1}^{3}z}{t_{1}^{2}}\right|=\infty
\]
since $\lim_{z\rightarrow-\infty}t_{1}=0$. So 
\[
\limsup_{z\rightarrow-\infty}|t_{1}^{4}z^{2}|=\limsup_{z\rightarrow-\infty}|(t_{1}^{3}z)(t_{1}z)|=\infty.
\]
We factor $t_{1}^{4}z^{2}$ out of $D(t_{1},z)$ and utilize $\limsup_{z\rightarrow-\infty}|t_{1}z|=\limsup_{z\rightarrow-\infty}|t_{1}^{2}z|=\infty$
to get the following 
\[
\limsup_{z\rightarrow-\infty}|D(t_{1},z)|=\limsup_{z\rightarrow-\infty}|t_{1}^{4}z^{2}|\cdot\left|1-\frac{2}{t_{1}z}-\frac{36}{t_{1}z^{2}}+\frac{49}{t_{1}^{2}z^{2}}-\frac{2}{t_{1}^{2}z}-\frac{14}{t_{1}^{2}z(t_{1}z)}+\frac{1}{t_{1}^{4}z^{2}}\right|=\limsup_{z\rightarrow-\infty}|t_{1}^{4}z^{2}|=\infty.
\]
This contradicts to $D(t_{1},z)=0$ as $t_{1}$ is a zero of $D(t,z)$. 

The equation $\lim_{z\rightarrow-\infty}|t_{1}^{3}z|=0$ allows us
to rewrite $\lim_{z\rightarrow-\infty}D(t_{1},z)$ as the following
\[
0=\lim_{z\to-\infty}D(t_{1},z)=\lim_{z\to-\infty}t_{1}^{4}z^{2}-2zt_{1}^{2}+1.
\]
If $u=\limsup_{z\rightarrow-\infty}t_{1}^{2}z$, then 
\[
0=u^{2}-2u+1=(u-1)^{2}
\]
from which we deduce that $u=1$. Similarly $\liminf_{z\rightarrow-\infty}t_{1}^{2}z=1$.
Thus $\lim_{z\rightarrow-\infty}t_{1}^{2}z=1$ from which we have
\[
\Arg\left(\lim_{z\rightarrow-\infty}t_{1}^{2}z\right)=0.
\]
Since $z$ is a negative real number, the equation above gives $\lim_{z\rightarrow-\infty}\Arg(t_{1})=\pi/2$
(recall that $\Arg(t_{1})>0$ as $t_{1}$ lies in the upper half plane).
We conclude the proof of this lemma. 
\end{proof}
\begin{rem}
From the proof of Lemma \ref{lem:decreasing}, we have 
\[
\lim_{z\rightarrow-\infty}t_{1}^{2}z=1
\]
from which and the fact that $t_{1}$ lies on the upper-half plane,
we deduce that as $z\rightarrow-\infty$
\[
t_{1}\sim\frac{i}{\sqrt{-z}}.
\]
We will obtain a more precise asymptotic approximation of $t_{1}$,
which will be useful later in the proof of Lemma \ref{lem:numzeros}.
We let 
\[
t_{1}=\frac{i}{\sqrt{-z}}+\epsilon
\]
where 
\[
\epsilon=o\left(\frac{1}{\sqrt{-z}}\right)
\]
and substitute this equation to $D(t_{1},z)=0$ to conclude
\begin{align*}
0 & =\epsilon^{4}z^{2}-4i\epsilon^{3}\sqrt{-z}z-2\epsilon^{3}z-36\epsilon^{3}+4\epsilon^{2}z+6i\epsilon^{2}\sqrt{-z}+\frac{108i\epsilon^{2}\sqrt{-z}}{z}\\
 & +49\epsilon^{2}-\frac{98i\epsilon\sqrt{-z}}{z}-\frac{108\epsilon}{z}-20\epsilon+\frac{36i\sqrt{-z}}{z^{2}}+\frac{16i\sqrt{-z}}{z}+\frac{49}{z}.
\end{align*}
We apply $\epsilon=o(1/\sqrt{-z})$ and reduce this identity to 
\[
0=4\epsilon^{2}z+\frac{16i\sqrt{-z}}{z}+o\left(\epsilon^{2}|z|+\frac{1}{\sqrt{|z|}}\right)
\]
from which we conclude 
\[
\epsilon=\pm\frac{2e^{3i\pi/4}}{(-z)^{3/4}}+o\left(\frac{1}{|z|^{3/4}}\right).
\]
Consequently 
\[
t_{1}=\frac{i}{\sqrt{-z}}\pm\frac{2e^{3i\pi/4}}{(-z)^{3/4}}+o\left(\frac{1}{|z|^{3/4}}\right).
\]
Since $t_{1}$ lies in the first quadrant, we conclude that
\begin{equation}
t_{1}=\frac{i}{\sqrt{-z}}-\frac{2e^{3i\pi/4}}{(-z)^{3/4}}+o\left(\frac{1}{|z|^{3/4}}\right).\label{eq:t1approx}
\end{equation}
 
\end{rem}

Recall that we want to find the number of zeros of $H_{m}(z)$ on
the interval $(-\infty,-9/4)$. To achieve this goal, we will provide
a closed formula for the polynomial $H_{m}(z)$. The following lemma
provides such a formula in terms of $t_{1}(z)$, $t_{2}(z)$, $t_{3}(z)$,
and $t_{4}(z)$. For the ease of notations, we suppress the parameter
$z$ in these variables. 
\begin{lem}
\label{lem:partialfrac}For any $z\in(-\infty,-9/4)$, if $t_{1}$,
$t_{2}$, $t_{3}$ and $t_{4}$are the zeros of (\ref{eq:denominator}),
then 
\begin{equation}
-z^{2}P_{m}(z)=\frac{A}{t_{1}^{m+1}}+\frac{B}{t_{2}^{m+1}}+\frac{C}{t_{3}^{m+1}}+\frac{D}{t_{4}^{m+1}}\label{eq:Hmformula}
\end{equation}
where 
\begin{align*}
A & =\frac{1}{(t_{1}-t_{2})(t_{1}-t_{3})(t_{1}-t_{4})},\\
B & =\frac{1}{(t_{2}-t_{1})(t_{2}-t_{3})(t_{2}-t_{4})},\\
C & =\frac{1}{(t_{3}-t_{1})(t_{3}-t_{2})(t_{3}-t_{4})},\\
D & =\frac{1}{(t_{4}-t_{1})(t_{4}-t_{2})(t_{4}-t_{3})}.
\end{align*}
\end{lem}

\begin{proof}
Since $t_{1},t_{2},t_{3},$ and $t_{4}$ are the zeros of the denominator
on the right side of (\ref{eq:Pmgen}), we can factor this denominator
and write our generating function as 
\[
\sum_{m=0}^{\infty}P_{m}(z)t^{m}=\frac{1}{z^{2}(t-t_{1})(t-t_{2})(t-t_{3})(t-t_{4})}.
\]
As $t_{1}$, $t_{2}$, $t_{3}$, and $t_{4}$ are distinct by Lemma
\ref{lem:distinct}, partial fraction decomposition yields 
\[
\sum_{m=0}^{\infty}P_{m}(z)=\frac{1}{z^{2}}\left(\frac{A}{t-t_{1}}+\frac{B}{t-t_{2}}+\frac{C}{t-t_{3}}+\frac{D}{t-t_{4}}\right)
\]
where $A$, $B$, $C$, and $D$ are given in the statement of the
lemma. We express each term on the right side as a power series in
$t$ as follow: 
\[
\frac{A}{t-t_{1}}=\frac{A}{t_{1}(1-\frac{t}{t_{1}})}=\sum_{m=0}^{\infty}-\frac{At^{m}}{t_{1}^{m+1}}.
\]
From similar computations for the remainder three terms, we conclude
\[
-z^{2}P_{m}(z)=\frac{A}{t_{1}^{m+1}}+\frac{B}{t_{2}^{m+1}}+\frac{C}{t_{3}^{m+1}}+\frac{D}{t_{4}^{m+1}},
\]
from which our lemma follows. 
\end{proof}
We note that the first two terms on the right side of \eqref{eq:Hmformula}
are complex conjugates and the same statement holds for the last two
terms of this expression. To count the number of real zeros of $P_{m}(z)$,
we will find the dominant term on the right side of \eqref{eq:Hmformula},
which is provided by the lemma below. 
\begin{lem}
\label{lem:compare} For any $m\in\mathbb{N}$ and any $z\in(-\infty,-9/4)$,
let $A,B,C,D$ be defined as in Lemma \ref{lem:partialfrac}. Then
\[
\left|\frac{A}{t_{1}^{m+1}}\right|>\left|\frac{C}{t_{3}^{m+1}}\right|.
\]
\end{lem}

\begin{proof}
It is equivalent to show 
\[
\left|\frac{t_{3}}{t_{1}}\right|^{m+1}>\left|\frac{C}{A}\right|.
\]
From the definition of $A$ and $C$ in \eqref{eq:Hmformula}, the
right side is 
\[
\left|\frac{C}{A}\right|=\left|\frac{(t_{1}-t_{2})(t_{1}-t_{4})}{(t_{3}-t_{2})(t_{3}-t_{4})}\right|.
\]
Since $\overline{t_{1}}=t_{2}$ and $\overline{t_{3}}=t_{4}$, we
have $|t_{1}-t_{4}|=|t_{3}-t_{2}|$ and consequently the right hand
side becomes 
\[
\left|\frac{t_{1}-t_{2}}{t_{3}-t_{4}}\right|=\left|\frac{\Im(t_{1})}{\Im(t_{3})}\right|.
\]
Thus it remains to prove that 
\[
\left|\frac{t_{3}}{t_{1}}\right|^{m+1}>\left|\frac{\Im(t_{1})}{\Im(t_{3})}\right|.
\]
Since $|t_{3}|>|t_{1}|$ by Lemma \ref{lem:zerostrictineq}, it suffices
to show 
\[
\Im(t_{3})>\Im(t_{1}).
\]
From the fact that $|t_{3}|>|t_{1}|$ and $t_{1}$ and $t_{3}$ lie
on the upper-half plane, it remains to show $\sin(\Arg(t_{3}))>\sin(\Arg(t_{1}))$.
We recall that $\theta=\Arg(t_{1})$ and $\phi=\Arg(t_{3})$. By the
continuity of $\theta$ and $\phi$ as functions of $z$, it suffices
to show $\sin\phi\ne\sin\theta$ for all $z\in(-\infty,-9/4)$ and
verify this inequality at one value of $z$. It is easy to check that
with a computer algebra system that at $z=-5$,, $\theta=0.37..$
and $\phi=1.299..$.

To finish the proof of this lemma we will prove $\sin(\phi)\ne\sin(\theta)$.
Assuming by contradiction that $\sin(\phi)=\sin(\theta)$, which is
equivalent to either $\phi=\theta$ or $\phi=\pi-\theta$.In the first
case when $\phi=\theta$, Equations \eqref{eq:firstsym} and \eqref{eq:thirdfourthsym}
give 
\[
-\frac{14}{z}=\frac{2z+36}{z^{2}}
\]
which implies $z=-\frac{9}{4}$. This contradicts to $z\in(-\infty,-\frac{9}{4})$.

Similarly, in the the case $\phi=\pi-\theta$ or equivalently $\cos\theta=-\cos\phi$,
we have 
\[
\frac{14}{z}=\frac{2z+36}{z^{2}},
\]
which implies $z=3$, a contradiction to the fact that $z\in(-\infty,-\frac{9}{4})$. 
\end{proof}
Recall that we want to find the number of zeros of $P_{m}(z)$ on
the interval $(-\infty,-9/4)$ and compare that number with the degree
of this polynomial which is at most $\lfloor m/2\rfloor$ by Lemma
\ref{lem:degree}. The lemma below gives a lower bound of the number
of real zeros of $P_{m}(z)$ on the given interval. Theorem \ref{thm:maintheorem}
follows from this lemma and the Fundamental Theorem of Algebra.

\begin{figure}
\begin{centering}
\includegraphics[scale=0.5]{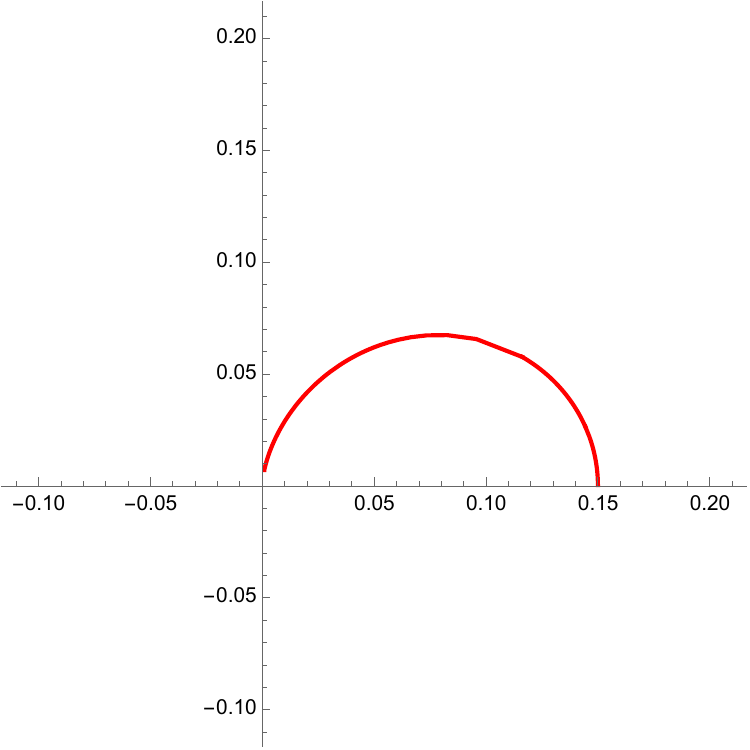} 
\par\end{centering}
\caption{\label{fig:t1z}The curve $t_{1}(z),-\infty<z<-9/4$}
\end{figure}

\begin{lem}
\label{lem:numzeros}$P_{m}(z)$ has at least $\lfloor\frac{m}{2}\rfloor$
zeros on the interval $(-\infty,-9/4)$. 
\end{lem}

\begin{proof}
Recall from \eqref{eq:Hmformula} that for $z\in(-\infty,-9/4)$
\begin{equation}
-z^{2}P_{m}(z)=2\Re\left(\frac{A}{t_{1}^{m+1}}\right)+2\Re\left(\frac{C}{t_{3}^{m+1}}\right)\label{eq:P_mReform}
\end{equation}
where, by Lemma \ref{lem:compare}, 
\[
\left|\frac{A}{t_{1}^{m+1}}\right|>\left|\frac{C}{t_{3}^{m+1}}\right|.
\]
Let 
\begin{equation}
f(t_{1})=\frac{A}{t_{1}^{m+1}}.\label{eq:fdef}
\end{equation}
From the Implicit Function Theorem, the function $t_{1}(z)$, $z\in(-\infty,-9/4)$
produces a smooth curve in the complex plane (see Figure \ref{fig:t1z})
and hence so is $f(t_{1}(z))$. We deduce from (\ref{eq:P_mReform})
that at the value of $z$ where $f(t_{1}(z))\in\mathbb{R}^{+}$ or
$f(t_{1}(z))\in\mathbb{R}^{-}$ we have $-z^{2}P_{m}(z)>0$ or $-z^{2}P_{m}(z)<0$
respectively. By the Intermediate Value theorem, there is a zero of
$P_{m}(z)$ when there is a change in sign of $-z^{2}P_{m}(z)$. Thus,
to count the number of zeros of $P_{m}(z)$ on $(-\infty,-9/4)$,
we can count number of times the curve $f(t_{1}(z))$, $z\in(-\infty,-9/4)$,
intersects the real axis. To count this number, we compute the change
in argument of this curve. From (\ref{eq:fdef}) we have 
\begin{equation}
\Delta\arg_{-\infty<z<-9/4}f(t_{1}(z))=\Delta\arg_{-\infty<z<-9/4}A-(m+1)\Delta\arg_{-\infty<z<-9/4}t_{1}(z).\label{eq:changeargft_1}
\end{equation}
We deduce from Lemma \ref{lem:decreasing} that 
\[
\Delta\arg_{-\infty<z<-9/4}t_{1}(z)=-\pi/2.
\]
To measure the change in argument of $A$, we claim that $A\notin i\mathbb{R}^{+}$
for all $z\in(-\infty,-9/4)$ from which we can conclude that the
change in argument is 
\begin{equation}
\lim_{z\rightarrow-9/4}\Arg A-\lim_{z\rightarrow-\infty}\Arg A.\label{eq:changeargA}
\end{equation}
Indeed, we write 
\begin{align*}
t_{1} & =a+bi,\\
t_{2} & =a-bi,\\
t_{3} & =c+di,\\
t_{4} & =c-di,
\end{align*}
and obtain from procedural computations that 
\[
A=\frac{1}{(t_{1}-t_{2})(t_{1}-t_{3})(t_{1}-t_{4})}=\frac{1}{(2bi)(a^{2}-b^{2}+c^{2}+d^{2}-2ac+2abi-2bci)}.
\]
If by contradiction $A\in i\mathbb{R}^{+}$, then the fact that $b>0$
(as $t_{1}$ lies in the first quadrant) gives 
\[
a^{2}-b^{2}+c^{2}+d^{2}-2ac+2abi-2bci\in\mathbb{R^{-}}.
\]
which implies $b(a-c)=0$. In the first case when $b=0$, the real
part of the expression above is $(a-c)^{2}+d^{2}>0$. In the second
case when $a-c=0$, this real part is $d^{2}-b^{2}>0$ by Lemma \ref{lem:zerostrictineq}.
Thus we obtain a contradiction in both cases.

We now compute (\ref{eq:changeargA}). To compute the first term in
this expression, we note from 
\[
D(t,-9/4)=\frac{1}{16}\left(9t^{2}-28t+4\right)^{2},
\]
the equation $D(t_{1}(z),z)=0$, and Lemma \ref{lem:zerostrictineq},
that 
\[
\lim_{z\rightarrow-9/4}t_{1}(z)=\lim_{z\rightarrow-9/4}t_{2}(z)=\frac{2}{9}\left(7-2\sqrt{10}\right)
\]
and 
\[
\lim_{z\rightarrow-9/4}t_{3}(z)=\lim_{z\rightarrow-9/4}t_{4}(z)=\frac{2}{9}\left(7+2\sqrt{10}\right).
\]
We recall from Lemma \ref{lem:compare} that 
\[
A=\frac{1}{(t_{1}-t_{2})(t_{1}-t_{3})(t_{1}-t_{4})}.
\]
As $z\rightarrow-9/4$ we have 
\[
(t_{1}-t_{3})(t_{1}-t_{4})\rightarrow\mathbb{R}^{+}
\]
while 
\[
\Arg(t_{1}-t_{2})=\pi/2
\]
since $t_{2}=\overline{t_{1}}$ and $t_{1}$ lies in the upper half
plane. These equations imply that 
\begin{equation}
\lim_{z\rightarrow-9/4}\Arg A=-\frac{\pi}{2}.\label{eq:angAright}
\end{equation}

We will next compute the second term of (\ref{eq:changeargA}). We
note that the denominator of $A$ is $D_{t}(t_{1},z)/z^{2}$ since
\[
D(t,z)=z^{2}(t-t_{1})(t-t_{2})(t-t_{3})(t-t_{4}).
\]
Thus 
\[
\lim_{z\rightarrow-\infty}\Arg(A)=-\lim_{z\rightarrow-\infty}\Arg D_{t}(t_{1},z)
\]
where from (\ref{eq:denominator})
\[
D_{t}(t_{1},z)=4t_{1}^{3}z^{2}-6t_{1}^{2}z-108t_{1}^{2}-4t_{1}z+98t_{1}-14.
\]
We substitute (\ref{eq:t1approx}) to the right side of this equation
to obtain the following asymptotics as $z\rightarrow\infty$
\[
D_{t}(t_{1},z)\sim16(-z)^{1/4}e^{3i\pi/4}
\]
and consequently 
\[
\lim_{z\rightarrow-\infty}\Arg(A)=-3\pi/4.
\]
We conclude from (\ref{eq:angAright}) and the fact that $A\notin i\mathbb{R}$
that 
\[
\Delta\arg_{-\infty<z<-9/4}A=\pi/4.
\]
Consequently, from (\ref{eq:changeargft_1}) 
\[
\Delta\arg_{-\infty<z<-9/4}f(t_{1}(z))=\frac{\pi}{4}+\frac{(m+1)\pi}{2}=\frac{m\pi}{2}+\frac{3\pi}{4}.
\]
We recall that that at the values of $z\in(-\infty,-9/4)$ where $f(t_{1}(z))\in\mathbb{R}^{+}$
or $f(t_{1}(z))\in\mathbb{R}^{-}$ we have $-z^{2}H_{m}(z)>0$ or
$-z^{2}H_{m}(z)<0$ respectively. By the Intermediate Value Theorem,
there is at least a zero of $H_{m}(z)$ between two consecutive values
of $z$ where $f(t_{1}(z))$ changes from $\mathbb{R}^{+}$ to $\mathbb{R}^{-}$
or vice versa. Since the change of argument of $f(t_{1}(z))$ is $m\pi/2+3\pi/4$
there are at least $\lfloor\frac{m}{2}\rfloor$ such changes. Thus
we obtain at least $\lfloor\frac{m}{2}\rfloor$ zeros of $P_{m}(z)$
on $(-\infty,-9/4)$ and the lemma follows. 
\end{proof}


\begin{thebibliography}{10}
\bibitem[1]{aar}Andrews GE, Askey R, Roy R. Special Functions. Cambridge
University Press; 1999.

\bibitem[2]{apostol}T. M. Apostol, The resultants of the cyclotomic
polynomials Fm(ax) and Fn(bx), Math. Comp. 29 (1975), 1--6.

\bibitem[3]{ds}K. Dilcher and K. B. Stolarsky, Resultants and discriminants
of Chebyshev and related polynomials, Trans. Amer. Math. Soc. 357
(2005), no. 3, 965--981.

\bibitem[4]{lt}J. Luong, K. Tran, Zeros of a table of polynomials
satisfying a four-term contiguous relation, Z. Anal. Anwend. (2022),
DOI 10.4171/ZAA/1698.

\bibitem[5]{eft}M. Erickson, S. Fernando, K. Tran, Enumerating rook
and queen paths, Bull. Inst. Combin. Appl. 60 (2010), 37--48.

\bibitem[6]{ft}T. Forgacs, K. Tran, Hyperbolic polynomials and linear-type
generating functions, J. Math. Anal. Appl. Volume 488, Issue 2, 15
August 2020.

\bibitem[7]{gi}J. Gishe and M. E. H. Ismail, Resultants of Chebyshev
polynomials, Z. Anal. Anwend. 27 (2008), no. 4, 499--508.

\bibitem[8]{rees}E. L. Rees (1922) Graphical Discussion of the Roots
of a Quartic Equation, The American Mathematical Monthly, 29:2, 51-55,
DOI: 10.1080/00029890.1922.11986100 .

\bibitem[9]{sw}E. Stein, G. Weiss, Introduction to Fourier Analysis
on Euclidean Spaces, Princeton, N.J.: Princeton University Press,
1971, ISBN 978-0-691-08078-9.

\bibitem[10]{tz}Tran, K., Zumba, A. Zeros of polynomials with four-term
recurrence and linear coefficients. Ramanujan J 55, 447--470 (2021).
\end{thebibliography}
\end{document}